\RequirePackage[l2tabu, orthodox]{nag}

\documentclass[reqno]{amsart}

\usepackage{lmodern}
\usepackage[T1]{fontenc}
\usepackage[utf8]{inputenc}
\usepackage[english]{babel}
\usepackage{microtype} %Better font spacing
\usepackage{comment}
\usepackage{varwidth} %for tableaux

\usepackage{amsmath,amssymb,amsthm,mathrsfs,listings,xskak,
latexsym,mathtools,mathdots,booktabs,enumitem,bm,url}
\usepackage[centertableaux]{ytableau}
\usepackage[pdftex]{hyperref}
\usepackage[capitalize,noabbrev]{cleveref}
\usepackage{caption}
\usepackage{subcaption}

\usepackage{todonotes}
\presetkeys%
    {todonotes}%
    {inline,backgroundcolor=yellow}{}

\newtheorem{theorem}{Theorem}
\newtheorem{proposition}[theorem]{Proposition}
\newtheorem{lemma}[theorem]{Lemma}

\newtheorem{conjecture}[theorem]{Conjecture}

\theoremstyle{definition}
\newtheorem{definition}[theorem]{Definition}
\newtheorem{example}[theorem]{Example}
\newtheorem{remark}[theorem]{Remark}

\definecolor{lightblue}{rgb}{0.8,0.8,1.0}
\definecolor{lightgreen}{rgb}{0.8,1.0,0.8}
\definecolor{pBlue}{RGB}{86,139,190}
\definecolor{pCyan}{RGB}{149,186,201}
\definecolor{pSand}{RGB}{184,166,121}
\definecolor{pAlgae}{RGB}{87,115,135}
\definecolor{pSkin}{RGB}{236,216,167}
\definecolor{pGray}{RGB}{156,175,156}
\definecolor{pPink}{RGB}{215,114,127}
\definecolor{pOrange}{RGB}{211,153,80}

\tikzset{
dot/.style = {circle, fill, minimum size=#1,
              inner sep=0pt, outer sep=0pt},
dot/.default = 6pt % size of the circle diameter 
}

\newcommand{\defin}[1]{%
\relax\ifmmode%
\textcolor{blue}{#1}%
\else\textcolor{blue}{\emph{#1}}%
\fi%
}

\newcommand{\thsup}{\textnormal{th}}

\renewcommand{\rook}{\raisebox{-0.1em}{\symrook}}

\newcommand{\setZ}{\mathbb{Z}}

\tikzset{every picture/.append
	style={
		scale=1,
		x=1em,
		y=1em,
		entries/.style={xshift=-0.5em,yshift=-0.5em,font=\small},
		thickLine/.style={line width=1.4pt,line join=round},
		bgEntry/.style={xshift=-0.5em,yshift=-0.5em,
			regular polygon,regular polygon sides=4,fill,inner sep=0pt,minimum size=1.35em
		}
	}
}
\usetikzlibrary{calc}%for coordinate calculation in tikz
\DeclareMathOperator{\des}{des}
\DeclareMathOperator{\peak}{peak}

\DeclareMathOperator{\id}{id}

\DeclareMathOperator{\exc}{exc}
\DeclareMathOperator{\asc}{asc}

\DeclareMathOperator{\Asc}{Asc}

\newcommand{\symS}{\mathfrak{S}} %Permutations
\makeatletter
\newcommand{\preceqdot}{\mathrel{\mathpalette\pr@ceqd@t\relax}}
\newcommand{\pr@ceqd@t}[2]{%
  \begingroup
  \sbox\z@{$#1\prec$}\sbox\tw@{$#1\preceq$}%
  \dimen@=\dimexpr\ht\tw@-\ht\z@\relax
  {\preceq}%
  \mkern-5mu
  \raisebox{\dimen@}{$\m@th#1\cdot$}%
  \endgroup
}
\makeatother

\newcommand{\xvec}{\mathbf{x}}

\newcommand{\svec}{\mathbf{s}}
\newcommand{\evec}{\mathbf{e}}

\title{Real-rootedness of rook-Eulerian polynomials}

\author[Per Alexandersson]{Per Alexandersson}
\address{Department of Mathematics, Stockholm University, SE-106 91 Stockholm, Sweden}
\email{per.w.alexandersson@gmail.com}

\author[Aryaman Jal]{Aryaman Jal}
\address{Department of Mathematics, KTH Royal Institute of Technology
SE-100 44 Stockholm, Sweden}
\email{aryaman@kth.se}

\author[Maena Quemener]{Maena Quemener}
\address{Department of Mathematics, École normale supérieure Paris-Saclay
4, avenue des Sciences
91190, Gif-sur-Yvette}
\email{maena.quemener@ens-paris-saclay.fr}


\begin{document}

\begin{abstract}
   We introduce rook-Eulerian polynomials, a generalization of the classical Eulerian polynomials arising from complete rook placements on Ferrers boards, and prove that they are real-rooted. We show that a natural context in which to interpret these rook placements is as lower intervals of $312$-avoiding permutations in the Bruhat order. We end with some variations and generalizations along this theme.   
\end{abstract}

\maketitle 
In this note, we study rook-Eulerian polynomials, a generalization of the classical Eulerian polynomials, which arise from complete rook placements
on Ferrers boards. Using the framework of sequences of interlacing polynomials,
we prove that rook-Eulerian polynomials are real-rooted. We then revisit a folklore fact: complete rook placements on Ferrers boards have a natural poset structure, namely the lower order ideal of a $312$-avoiding permutation in the Bruhat order. The $312$-avoiding nature of the permutation 
is necessary for the real-rootedness to hold. We contrast rook-Eulerian polynomials with $\svec$-Eulerian polynomials, and show that the two families overlap without coinciding with one another. Finally we consider descent-generating polynomials of lower intervals in the weak order on $\symS_{n}$ and show that while they are not real-rooted in general, are likely to have ultra-log-concave coefficients.

\section{Preliminaries}

A \defin{board} $B$ is a subset of the square grid $[n] \times [n]$. By a \defin{non-attacking} placement of rooks on a board, we mean a subset $S \subset B$ such that no two elements of $S$ share the same $x$-coordinate or $y$-coordinate. This corresponds to a placement of rooks on $B$ such that no two rooks
can take each other in the sense of chess. Hereafter, the boards in question will arise from partitions $\lambda = (\lambda_{1}, \ldots , \lambda_{n})$ with $\lambda_{1} \leq \lambda_{2} \leq \ldots \leq \lambda_{n}$. We will use the words partition and \defin{Ferrers board} interchangeably. We use the French convention for drawing Ferrers boards: left-justified arrays 
of boxes are drawn such that the first has $\lambda_{1}$ cells, the second below it has $\lambda_{2}$ cells and so on. The columns of the board are labeled from $1$ to $n$ from left to right.

\begin{definition}
A \defin{complete rook placement} on a Ferrers board is a rook placement such 
that each column and row contains exactly one rook. 
    
We note that in order to have at least one  complete rook placement, $\lambda$ must satisfy $\lambda_{i} \geq i$ for all $i \in [n]$. It will be useful to view complete rook placements as permutations. To that end, given a complete rook placement on a shape $\lambda$ with $n$ parts, define a 
permutation $\sigma \in \symS_n$, with $\sigma_{i}$ equal to the index 
of the column of the rook placed in the $i^\thsup$ row. See Figure~\ref{fig:complete_rook_perm_eg} for an example. 
\end{definition}

The full permutation group carries two natural and extensively studied poset structures that we recall below.  

\begin{figure}[!ht]
\ytableausetup{boxsize=1.1em}
\[    
\begin{ytableau}
 \none[1] & \none[2] & \none[3] &  \none[4] \\
 \, &\rook \, & & *(lightgray) \\
 &  \, & \, &\rook \,\\
   & & \, \rook & \, \\
   \rook & \, & \, & \, \\
\end{ytableau}
\]
\caption{The complete rook placement $\sigma = 2431$ on $\lambda = 3444$.}
\label{fig:complete_rook_perm_eg}
\end{figure}

\begin{definition}[Weak and strong order]
The \defin{weak order} on $\symS_n$ is a partial order,
where the cover relations are defined as follows.
The permutation $\pi$ covers $\sigma$
if $\sigma$ can be obtained from $\pi$ by an \emph{adjacent}
transposition that decreases the number of inversions. An adjacent transposition refers to a swap of adjacent elements in $\sigma$. 
We write \defin{$\sigma \leq_W \pi$} if $\sigma$
is less than or equal to $\pi$ in the weak order.

Similarly, the \defin{strong order} or \defin{Bruhat order} is 
defined in the same manner, but here the transposition
need not be adjacent.
We denote this relation by \defin{$\sigma \leq_B \pi$}.

The maximal element in both orders 
is unique, and given by the \defin{longest permutation}
\defin{$\omega_0 = [n,\dotsc,2,1]$}.
\end{definition}

Recall that a permutation $\sigma$ contains  the pattern $312$ if there exists $i<j<k$ such that $\sigma_{j}<\sigma_{k}<\sigma_{i}$. If this does not hold, we say that $\sigma$ is~\defin{$312$-avoiding}. 

The following proposition is a folklore fact in Catalan combinatorics; we include a proof for completeness. Rows and columns of $B$ are indexed by coordinates $(i, j) \in [n] \times [n]$. A pair of rooks $(i, j)$ and $(k, \ell)$ is \defin{non-nested} if $i<k$ and $\ell < j$, and nested otherwise. In terms of rooks on the board, non-nesting means means that one rook does not occur South-East of another. Given a pair of non-nested rooks $(i, j)$ and $(k, \ell)$ define the \defin{switch move}\footnote{Develin uses the same terminology but with $\lambda$ drawn according to the English convention~\cite{Develin2006RookEquiv}.} on rook $(i, j)$ to be the swap that replaces the given pair of rooks with $(k, j)$ and $(i, \ell)$. Similarly, define the \defin{inverse switch move} to be the swap that replaces a pair of nested rooks with their non-nesting counterpart.   

\begin{proposition}\label{prop:312_to_complete_rook}
Every $312$-avoiding permutation $\sigma \in \symS_n$ corresponds 
to a Ferrers board $\lambda = (\lambda_1,\dotsc,\lambda_n)$ with $\lambda_i \geq i$ for $i=1, \ldots, n-1$ and $\lambda_{n} = n$, and vice-versa.
The permutations in the interval 
\[
\{\pi \in \symS_{n}: \mathrm{id} \leq_{B} \pi \leq_{B} \sigma \}
\]
are exactly the complete rook placements on the Ferrers board $\lambda$.
\end{proposition}
\begin{proof}
For $n \geq 1$, let $\mathfrak{S}_{312}$ and $\Lambda$ respectively be the set of $312$-avoiding permutations on $[n]$, and the Ferrers diagrams $\lambda$ fitting inside an $[n] \times [n]$ grid satisfying $\lambda_{i} \geq i$. 
We define a bijection from $\mathfrak{S}_{312}$ to $\Lambda$ as follows. Given $\sigma \in \mathfrak{S}_{312}$ define a 
complete rook placement $\rho$ on $[n] \times [n]$ by 
\[
    \rho =\{(i, \sigma_{i}): i \in [n]\}
\]
and let $\lambda$ be the 
Ferrers boards defined such that the last cell of every row of $\lambda$ either contains a rook of $\rho$ or has a rook of $\rho$ above it in the same column. 
The fact that $\sigma$ is a permutation immediately implies that $\lambda_{i} \geq i$.
Further, every such $\lambda$ arises from a complete rook placement: given $\lambda$, define $\rho'$ by placing 
a rook in $(1, \lambda_{1})$ and for $i \in [2, n]$ place a rook as far to the right as possible such that it does not attack the rooks that have already been placed in rows $1$ to $i-1$. This is clearly a complete rook placement on $\lambda$; let $\sigma'$ be the permutation corresponding to $\rho'$.

\textbf{Claim:} The permutation $\sigma$ avoids $312$. 

To see this, note that the presence of a $312$ pattern in $\sigma'$ corresponds to three rooks in $\rho'$ 
in a configuration as in Figure~\ref{fig:312_as_rooks} (after deleting the rows and columns not involved in the $312$ pattern). This is a contradiction since by construction, 
the rooks in each row of $\rho'$ are placed as far to the right as possible without attacking the rooks above that row.

For the second claim, fix $\sigma \in \symS_{312}$ and let $\mathcal{R}_{\lambda}$ be the set of complete rook placements on $\lambda$, where $\rho_{\max}$ is the rook placement corresponding to $\sigma$. Let $f:[\mathrm{id}, \sigma]_{B} \to \mathcal{R}_{\lambda}$ be the map that associates a permutation $\pi$ with its complete rook placement $\rho_{\pi}$ on $[n] \times [n]$. We need to show that $\rho_{\pi}$ lies on $\lambda$. Since $\pi \leq_{B} \sigma$, $\pi$ can be obtained from $\sigma$ by a series of transpositions. For every such transposition $(i, j)$, perform the corresponding switch move to the rooks in rows $i$ and $j$ in the rook placement. The resulting rook placement is $\rho_{\pi}$. This shows that $\rho_{\pi}$ lies in $\mathcal{R}_{\lambda}$. In the opposite direction, given a rook placement $\rho$ on $\lambda$, let the corresponding permutation be $\pi$. We need to show that $\pi \leq_{B} \sigma$. We will first show that we can obtain $\rho_{\max}$ from $\rho$ as follows: for every row $i$ from $1$ to $n$, if $\rho$ has no rook in $(i, \sigma_{i})$, perform the inverse switch  the nested rooks $(i, \pi_{i})$ and $(j, \sigma_{i})$ where $j>i$ is the row index containing a rook of $\rho$ in column $\sigma_{i}$. (This pair of rooks is indeed nested because $\pi_{i}< \sigma_{i}$, which holds since the rook of $\rho_{\max}$ in row $i$ is placed as far to the right as possible.) The rook placement obtained after this procedure is clearly $\rho_{\max}$. By performing the same sequence of transpositions corresponding to the inverse switch moves, we can obtain $\sigma$ from $\pi$, which shows that $\pi\leq_{B} \sigma$. See Figure~\ref{fig:rookInterval} for an example of this correspondence. 
\end{proof}

One can show that the correspondence in the second half of the theorem holds more generally for complete rook placements on skew shapes,  see~\cite[Proposition 12]{Sjostrand2007} for instance.
    
\begin{figure}[!ht]
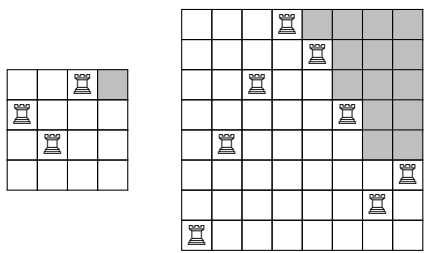

\ytableausetup{boxsize=1.1em}
\[    
\begin{ytableau} 
 \, &  \, &\rook & *(lightgray) \\
\rook &  \, & \, & \,\\
   &  \rook & \, & \, \\
   &  \, & \, & \, \\
\end{ytableau}
\qquad 
\begin{ytableau}
\; & \; & \; & \rook & *(lightgray)& *(lightgray)& *(lightgray)& *(lightgray) \\    
\; & \; & \; & \; &  \rook & *(lightgray)& *(lightgray)& *(lightgray) \\
\; & \; & \rook & \; &  \; & *(lightgray)& *(lightgray)& *(lightgray) \\
\; & \; & \; & \; &  \; & \rook & *(lightgray)& *(lightgray) \\
\; & \rook & \; & \; &  \; &  & *(lightgray)& *(lightgray) \\
\; & \; & \; & \; &  \; &  &  & \rook \\
 & \; & \; & \; &  \; &  & \rook & \; \\
\rook & \; & \; & \; &  \; &  &  & \; \\
\end{ytableau}
\]
\caption{Left: The $312$-pattern in a rook configuration.
Right:
The complete rook placement on the Ferrers board 
$45566888$ corresponds to the 312-avoiding permutation $45362871$. 
Every rook placement on this board can be obtained by switching pairs of non-nested rooks in the placement above. 
}\label{fig:312_as_rooks}
\end{figure}

The correspondence in Proposition~\ref{prop:312_to_complete_rook} 
is illustrated in Figure~\ref{fig:rookInterval}.
\begin{figure}[!ht]
\centering
\includegraphics[width=0.5\textwidth]{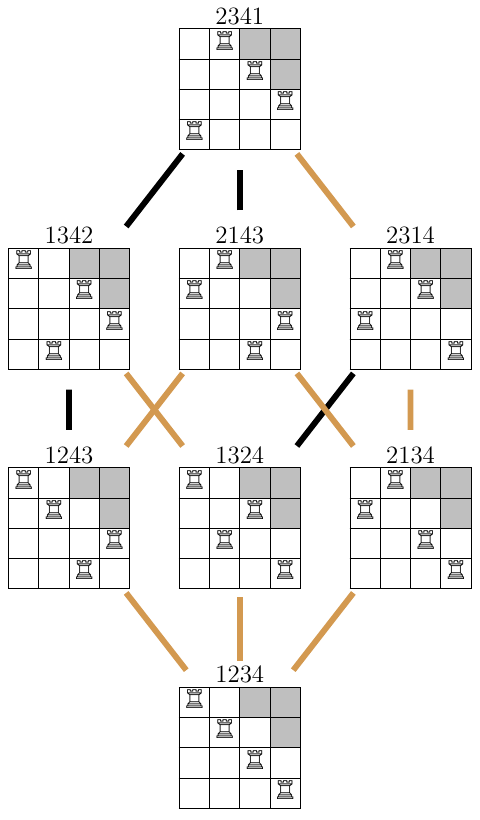}
\caption{The bijection between rook placements on a Ferrers board $\lambda = 2344$ 
and an interval $\{\sigma \in \symS_{4}: \text{id} \leq_{B} \sigma_{B} \leq 2341 \}$ in the Bruhat order. Black edges represent cover relations corresponding to arbitrary transpositions --- or applications of the switch move at the level of the rooks --- 
while yellow edges represent cover relations corresponding to adjacent transpositions.
}\label{fig:rookInterval}
\end{figure}

\begin{example}\label{eg:rook_perm_eg}
Below is a complete rook placement on the board $45566888$ with permutation $\sigma=31524687$.
\[
\begin{ytableau}
\; & \; & \rook & \; & *(lightgray)& *(lightgray)& *(lightgray)& *(lightgray) \\    
\rook & \; & \; & \; &  \; & *(lightgray)& *(lightgray)& *(lightgray) \\
\; & \; & \;& \; &  \rook& *(lightgray)& *(lightgray)& *(lightgray) \\
\; & \rook & \; & \; &  \; & \;& *(lightgray)& *(lightgray) \\
\; & \;& \; & \rook  &  \; &  & *(lightgray)& *(lightgray) \\
\; & \; & \; & \; &  \; & \rook &  &  \\
 & \; & \; & \; &  \; &  & \; & \rook  \\
\; & \; & \; & \; &  \; &  & \rook  & \; \\
\end{ytableau}
\]
\end{example}

\begin{definition}
An \defin{ascent} in a complete rook placement $\sigma$ is an index $i \in [n-1]$ such that $\sigma_i < \sigma_{i+1}$. 
In other words, the rook in row $i+1$ lies strictly to the right of the rook in row $i$.
We also extend the notion of ascents to words, i.e., if $w_1 w_2 \dotsc w_\ell$ is a word, then $i$ is an \defin{ascent of the word} if $w_i < w_{i+1}$.
\end{definition}

\begin{example}
The rook placement $\sigma = 31524687$ from Example~\ref{eg:rook_perm_eg} has $2,4,5$ and $6$ as ascents.
\end{example}

We denote by $\pi_{\lambda}$ the $312$-avoiding permutation corresponding to $\lambda$. By Proposition~\ref{prop:312_to_complete_rook}, the set of complete rook placements on $\lambda$ is the interval $[\id, \pi_{\lambda}]$ in the Bruhat order. The \defin{rook-Eulerian} polynomial of $\lambda$ is defined as
\[
    \defin{Q^{\lambda}(t)} \coloneqq \sum_{\sigma \in [\id,  \pi_{\lambda}]_{B}} t^{\asc(\sigma)}.
\]
When $\lambda = (n, \dots, n)$, the corresponding interval is the full permutation group $\symS_n$ and 
the associated $Q^\lambda$ is the Eulerian polynomial $A_{n}(t)$, the real-rootedness of which is well-known~\cite{Frobenius1910}.

This polynomial has arisen in the context of rook theory before. For example, $Q^{\lambda}(t)$ can be obtained as the $q=1$ specialization of the polynomial $t^{-1}A_{F}(t, q)$ 
considered in \cite{SongYan2012DescentsFerrers}, where $F$ is the Ferrers board $\lambda$. The main result of \cite{SongYan2012DescentsFerrers} states that $A_{F}(t,q)$ can be expressed as a 
permanent of a matrix the entries of which are area-generating polynomials of lattice paths. In a different context, evaluating $Q^{\lambda}$ at $t = 1$ and taking the square yields the number 
of Hamiltonian paths in the Ferrers graph corresponding to $\lambda$ \cite[Theorem 3.3]{EhrenborgVanWilligenburgFerrersGraphs2004}.

\section{Real-rootedness of rook-Eulerian polynomials}\label{section:rr_rook_eulerian}

In this section, we prove the real-rootedness of $Q^{\lambda}(t)$
and some generalizations of this family of polynomials. We first widen the scope of Ferrers boards under consideration. 

Suppose now that $\lambda$ satisfies $\lambda_{i} \geq i$ for $i=1, \ldots ,n$ and we do not require that $\lambda_{n} = n$ holds. A \defin{row-complete} rook placement on $\lambda$ is 
a placement of rooks such that every row contains a rook, and every column contains at most one rook.
Let \defin{$S_{\lambda}$} denote the set of row-complete rook placements on $\lambda$.

Similar to the previous section, a row-complete rook placement $\sigma$ is described as a \emph{word with distinct entries}  $\sigma_1\dotsc\sigma_n$, where $\sigma_{i}$ denotes the column index of the rook placed in the $i^\thsup$ row. The notion of ascent generalizes to this setting. By a minor abuse of notation, we let $Q^{\lambda}(t)$ denote the ascent-generating polynomial 
for row-complete rook placements on $\lambda$. We will refer to this as the \defin{rook-Eulerian polynomial on $\lambda$}. As we will only talk about row-complete rook placements in this section, we will sometimes refer to these objects only as ``rook placements.''

In order to prove real-rootedness of $Q^{\lambda}(t)$, we need to introduce a refinement.
Consider the set of rook placements on $\lambda$ 
such that column index of the rook placed in the first row is $i$. 
Its associated ascent-generating polynomial is
\begin{equation}\label{eq:ascQi}
    \defin{Q^{\lambda}_i(t)} \coloneqq \sum_{\substack{\sigma \in S_{\lambda} \\ \sigma_{1} = i}} t^{\asc(\sigma)}.
\end{equation}

By conditioning on the value of the first entry in the word, it is clear 
that these polynomials refine $Q^{\lambda}$ and that we have
\begin{equation}
Q^{\lambda}(t) = \sum_{i=1}^{\lambda_1} Q^{\lambda}_{i}(t).
\end{equation}

\begin{example}
Consider the shape $\lambda = (3,4,4,6,7)$. See the figure below for a row-complete rook placement on $\lambda$, encoded by the word $32417$, with ascents at positions $2$ and $4$.

\[
\begin{ytableau}
\ & \ & \rook \\
\ & \rook & \ &  \\
\ & \ & \ & \rook  \\
\rook & \ & \ & \ & \ & \ \\
\ & \ & \  & \ & \ & \ & \rook \\
\end{ytableau}
\]

The rook-Eulerian polynomial on $\lambda$ is $Q^\lambda(t) =  15 t^4+81 t^3+63 t^2+3 t$, while the polynomials refining $Q_{\lambda}$ are

\begin{align*}
Q^\lambda_1(t) &= 12t^2 + 30t^3 + 12t^4, \\
Q^\lambda_2(t) &= 21t^2 + 30t^3 + 3t^4, \\
Q^\lambda_3(t) &= 3t + 30t^2 + 21t^3.
\end{align*}
\end{example}

This refinement will be of use in obtaining the recursion below and applying the theory interlacing sequences of polynomials to this recursion.

\begin{definition}
Let $f$ and $g$ be polynomials with positive leading coefficients and real, non-positive
zeros, $a_{i}$ and $b_{i}$, respectively.
We say that $f$ \defin{interlaces} $g$, and we write $\defin{f \preceq g}$ if
\[
\dotsm \leq a_{3}\leq b_{3}\leq a_{2}\leq b_{2}\leq a_{1}\leq b_{1} \leq 0.
\]
Note that $\deg(f)=\deg(g)$ or $\deg(f)+1=\deg(g)$.
\end{definition}

We also need the concept of interlacing sequences, which has proven useful for establishing real-rootedness results. 

\begin{definition}
    A sequence $F = (f_{1}, \ldots ,f_{n})$ of real-rooted polynomials is \defin{interlacing} if $f_{i} \preceq f_{j}$ for $1\leq i < j \leq n$. 
\end{definition}

We first establish a simple recurrence for the polynomials $Q_{i}^{\lambda}$.

\begin{lemma}\label{lem:Q_lambda_recursion}
Let $\lambda = (\lambda_{1}, \dotsc , \lambda_{n})$ be a Ferrers board,
and let $\lambda^+$ be the the Ferrers board defined by $\lambda^+ \coloneqq (m,\lambda_{1}+1,  \dotsc , \lambda_{n}+1)$ for some $m \leq \lambda_{1}+1$.
Then for any $i \leq m$, the following holds:
\begin{equation}\label{eq:ascent_gen_recursion}
Q^{\lambda^+}_{i}(t) = 
\sum_{j=1}^{i-1} Q^{\lambda}_{j}(t)  + 
t \sum_{j=i}^{\lambda_1} Q^{\lambda}_j(t).
\end{equation}
\end{lemma}
\begin{proof}
Let $\sigma = i \, \sigma_{1}\ldots\sigma_{n}$ be a rook placement on $\lambda^+$.
The operation of deleting the first row of $\lambda^+$ followed by the $i^{\thsup}$ 
column of $\lambda$ bijectively maps rook placements on $\lambda^+$ with the first rook 
in column $i$, to rook placements on $\lambda$. 
If $\sigma_{1} = j$, then the number of ascents in the rook 
placement on $\lambda$ decreases by $1$ if and only if $j \geq i-1$,
and \eqref{eq:ascent_gen_recursion} follows immediately.
\end{proof}

\begin{example}
For example, if $\lambda^+ = (3,4,5,5,7,8)$ with $m=3$,
then $\lambda = (3,4,4,6,7)$. A rook placement on 
$\lambda^+$ arising from the term $Q_2^{\lambda}(t)$
is illustrated below:
\[
\begin{ytableau}
*(pOrange) & *(pOrange) \rook & *(pOrange) \\
\ & *(pOrange) & \ & \rook \\
\ & *(pOrange) & \rook & \ & \\
\ & *(pOrange) & \ & \ & \rook \\
\rook & *(pOrange) & \ & \ & \ & \ & \\
\ & *(pOrange) & \ & \ & \ & \ & \ & \rook \\
\end{ytableau}
\]
\end{example}

We now prove our main result: the sequence of polynomials 
\begin{equation}\label{eq:QSeq}
(
Q^{\lambda}_{\lambda_1}(t), \;
Q^{\lambda}_{\lambda_1-1}(t), \; \dotsc \;,
Q^{\lambda}_{2}(t),\;
Q^{\lambda}_{1}(t)
).
\end{equation}
is an interlacing sequence.

\begin{theorem}\label{prop:Q_lambda_real_rooted}
For all Ferrers boards of shape $\lambda = (\lambda_{1}, \ldots , \lambda_{n})$ and $n > 1$, 
the family of polynomials is $(Q^{\lambda}_i)_{i=1}^{\lambda_{1}}$ is interlacing. 
In particular, the rook-Eulerian polynomial $Q^{\lambda}$ is real-rooted.
\end{theorem}

\begin{proof}
We prove the first statement by induction on $n$, the size of the smallest $n \times n$ square containing the shape $\lambda.$ The base case is immediate since the polynomial associated to the $1 \times 1$ box is the constant polynomial $1$.
    
Now suppose $n \geq 2$ and that for all Ferrers boards $\lambda$ 
fitting inside an $n \times n$ square, the sequence of polynomials in
\eqref{eq:QSeq} is interlacing.  

Let $\lambda^+$ be a given Ferrers board fitting inside an $(n+1){\times}(n+1)$ square. Suppose $m$ is the first entry of $\lambda^{+}$ and $\lambda$ is the Ferrers shape obtained after deleting $m$ from $\lambda^{+}$ and then subtracting $1$ from each entry of the partition. The 
recursions in Lemma~\ref{lem:Q_lambda_recursion} can then be expressed using an $m{\times} \lambda_1$ matrix:
 \begin{gather}
    \begin{bmatrix}
    Q^{\lambda^+}_{m}\\
    \vdots\\
    \vdots\\
    Q^{\lambda^+}_{3} \\
    Q^{\lambda^+}_{2} \\
    Q^{\lambda^+}_{1}
    \end{bmatrix}
    =
    \begin{bmatrix}
        t & t & t & \dots & 1 & 1 & \dots & 1 & 1 \\
        t & t & t & \dots & t & 1 & \dots & 1 & 1 \\
        \vdots & \vdots & & & & \ddots & \ddots & \ddots & \vdots \\
        t & t & \dots & t & t & \dots & t & 1 & 1 \\
        t & t & \dots & t & t & \dots & t & t & 1 \\
        t & t & \dots & t & t & \dots & t & t &t
    \end{bmatrix}
   \begin{bmatrix}
    Q^{\lambda}_{\lambda_1}\\
    \vdots \\
    \vdots\\
    Q^{\lambda}_{3} \\
    Q^{\lambda}_{2} \\
    Q^{\lambda}_{1}
    \end{bmatrix}.
\end{gather}

Let $G_{\lambda}$ be the matrix above; its entries are given by\[
g_{i,j}^{\lambda}(t) = \begin{cases}
    t \quad &\text{if $1\leq j \leq \lambda_{1} - (i-1)$}\\
    1 \quad &\text{otherwise.}
\end{cases}
\]

This matrix satisfies the conditions of \cite[Corollary 8.7] {Branden2015}, and hence preserves the interlacing sequence property. Then, since the vector of polynomials on the right hand side of the equation above is interlacing (by the induction hypothesis), so is the left hand side. This completes the inductive step and hence we are done. The real-rootedness of $Q^{\lambda}$ follows immediately by noting that $t \cdot Q^{\lambda}(t) = Q^{\lambda^{+}}_{1}$, where $\lambda^{+} = (1, \lambda_{1}, \ldots , \lambda_{n})$.
\end{proof}

\subsection{A multivariate generalization}

One multivariate version of the rook-Eulerian polynomials can be defined as follows
\[
Q^\lambda(\xvec) = \sum_{\sigma \in S_{\lambda}} \prod_{i \in \Asc(\sigma)} x_i,
\]
where $\Asc(\sigma) = \{i \in [n-1]: \sigma_{i} < \sigma_{i+1}\}$ and $\xvec = (x_{1}, \ldots , x_{n-1})$. 

Recall that a multivariate polynomial $P(x_{1}, \dotsc , x_{n})$ is \defin{same-phase stable} 
if for all $\mathbf{\mu} \in \mathbb{R}_{+}^{n}$, the univariate polynomial $P(t\mathbf{\mu})$ is real-rooted. 
By an identical argument as in the proof of Proposition~\ref{prop:Q_lambda_real_rooted}, we have the following.
\begin{theorem}
For all Ferrers boards $\lambda$, $Q^\lambda(\xvec)$ is same-phase stable.
\end{theorem}
In the case of $\lambda=(n)^n$, same-phase stability was proved earlier in \cite{AlexanderssonNabawanda2021}, where it was also observed that stability of 
$Q^\lambda(\xvec)$ fails for $\lambda=(5,5,5,5,5)$.

\section{Variations and generalizations}

\subsection{Relation to \texorpdfstring{$s$}{s}-Eulerian polynomials}

Savage and Visontai introduced the $\svec$-Eulerian polynomials $E_{n}^{(\svec)}(x)$ in \cite{SavageVisontai2015} as the ascent-generating polynomial of $\svec$-inversion sequences. Among the various properties uncovered, they showed that this family of polynomials is real-rooted using a similar technique as in Section~\ref{section:rr_rook_eulerian}.   

Given that the $\svec$-Eulerian and rook-Eulerian polynomials are both real-rooted and 
satisfy similar-looking recurrences \cite[Lemma 2.3]{SavageVisontai2015}, 
it is natural to ask if there is equality between these families of polynomials. 
We show that rook-Eulerian polynomials are not included in the $\svec$-Eulerian polynomials. 
Thus, the polynomials $Q^{\lambda}$ are a new generalization of the Eulerian polynomials. 

\begin{definition}
For a sequence $\svec = (s_{1}, \ldots ,s_{n})$ of positive integers, 
the \defin{$\svec$-inversion sequences} are defined by:
 \[   
\mathcal{I}_n^{\svec} \coloneqq 
\{(e_1,\dots,e_n) \in \setZ^{n}: 0 \leq e_i < s_i \text{ for } 1 \leq i \leq n \}.
\]
    
The \defin{ascent set} of an $\svec$-inversion sequence $\evec = (e_1, \dots, e_n) \in \mathcal{I}_n^{(\svec)}$ is the set
\[
    \Asc(\evec) \coloneqq \left \{ i \in \{0, 1, \dots, n-1 \}: \frac{e_i}{s_i} < \frac{e_{i+1}}{s_{i+1}} \right\}
\]
with the convention that $e_0 = 0$ and $s_0 = 1$.

The \defin{ascent statistic} on $\evec \in \mathcal{I}_n^{(\svec)}$ is defined as:
\[
\asc(\evec) =|\Asc(\evec)|.
\]
\end{definition}

\begin{definition}
Given a sequence $\svec = (s_{1}, \ldots , s_{n})$ 
of positive integers, the $\svec$-Eulerian polynomial is defined as: 
\[
E_n^{(\svec)} (t) = \sum_{\evec \in \mathcal{I}_n^{(\svec)}} t^{\asc(\evec)}.
\]
\end{definition}

\begin{example}
    When $\svec = (1,2,\ldots , n)$, by a well-known bijection from inversion sequences to permutations, we get the $n^{\text{th}}$ Eulerian polynomial: 
    \[
    E_{n}^{(\svec)}(t) = A_{n}(t).
    \]
\end{example}

The following result shows that the family of rook-Eulerian polynomials is distinct from 
the family of the $\svec$-Eulerian polynomials.

\begin{theorem}\label{theorem:rook_Eulerian_not_s_Eulerian}
    Let $\lambda = (2,3,5,5,5)$,
    there is no sequence $\svec$ and integer $m$ such that~$Q^\lambda = E_m^{(\svec)}$. 
\end{theorem}

\begin{proof}

If to the contrary there exists $m$ and $\svec$ such that $Q^\lambda(t) = E_m^{(\svec)}(t)$, then 
evaluating both polynomials at $t=1$ yields $\prod_{i=1}^m s_i = \prod_{i=1}^n (\lambda_i - n + i)$. 
First note that deleting consecutive $1$s in $\svec$ yields the same $\svec$-Eulerian polynomial. 
This is because if $s_i = s_{i+1} = 1$ for some $0<i<m$, then each  $\evec \in \mathcal{I}_n^{(\svec)}$ 
satisfies $ e_i = e_{i+1} = 0$. In particular, $i \notin \text{Asc}(\evec)$. 
If we delete the consecutive $1$s from $\svec$ and denote the resulting sequence by $\svec'$, we have $E_m^{(\svec)} = E_{m-1}^{(\svec')}.$

This naturally places an upper bound on the possible values of $m$ and hence allows us to restrict 
ourselves to testing finitely many $\svec$ sequences as possible candidates for the required equality. 
By exhaustive computer search, there exists no $\svec$ satisfying $\prod_{i=1}^m s_i = \prod_{i=1}^n (\lambda_i - n + i)$ for $\lambda = (2,3,5,5,5)$. The corresponding $Q^{\lambda}$ is equal to
\[
Q^{\lambda}(t) = t^{3} + 9t^{2} + 13t + 1.
\]
\end{proof} 

\subsection{Descents in general Bruhat intervals}
Given the correspondence between complete rook placements and lower intervals of $312$-avoiding permutations in the Bruhat order, it is natural to ask if the real-rootedness holds for all lower intervals in the Bruhat order. That is, given $\pi \in \symS_n$, is
\[
Q^{\pi}(t) = \sum_{\sigma \in [\id, \pi]_{B}} t^{\des(\pi)}
\]
real-rooted? As the example below shows, the answer is no, thus highlighting the special nature of the $312$-avoiding condition. 

\begin{example}
    For $\pi = [4,6,2,1,7,3,5]$, using $\mathtt{Mathematica}$ one can compute the polynomial $Q^{\pi}$ to be
\begin{equation}
Q^{\pi}(t)  = 1 + 43 t + 196 t^2 + 168 t^3 + 23 t^4 + t^5,
\end{equation}

\noindent which has a pair of non-real roots near $-10.8143561 \pm 4.5913096i$.
For all permutations in $\symS_6$, we do get real-rooted polynomials by an exhaustive computer search. 
Note that, as expected, $[\mathbf{4}, 6, \mathbf{2}, 1, 7, \mathbf{3}, 5]$ contains the pattern $312$. 
\end{example} 

We can ask the same question for descent-generating polynomials of lower intervals $[\id, \sigma]_{W}$; this polynomial also fails to be real-rooted in general, but this time the counterexample has an interesting structure.

Recall that given a poset $P$ on $n$ elements, the Jordan--Hölder set $\mathcal{L}(P)$ of $P$, is the set of inverses of linear extensions of $P$, and $W_{P}$ is the descent-generating polynomial of $\mathcal{L}(P)$: \begin{align*}
\mathcal{L}(P) &= \{\sigma \in \symS_{n}: i \preceq_{P} j \implies \sigma_{i}^{-1} < \sigma_{j}^{-1}\},  \\
W_{P} &= \sum_{\sigma \in \mathcal{L(P)}}t^{\des(\sigma)}.
\end{align*}

Permutations naturally give rise to the following class of posets.

\begin{definition}\label{definition:permutation_poset}
    Let $\sigma \in \symS_{n}$. The partial order in the \defin{permutation poset} $P_{\sigma} = ([n], \preceq)$ is defined as \[
    i \preceq j \quad \iff \quad i \leq j \; \text{and} \; \sigma_{i}\leq \sigma_{j}.    \]
\end{definition}

The following lemma can be seen by interpreting the cover relation in the weak order in terms of containment of inversion sets.

\begin{lemma}\cite[Lemma 5]{FelsnerWernisch1997MarkovChains}\label{lem:permutation_poset_linear_extensions}
    The Jordan--Hölder set of $P_{\sigma}$ equals the lower interval $[\id, \sigma]_{W}$ in the weak order on $\symS_{n}$.
\end{lemma}

Recall that the \defin{width} of a poset is the size of the largest antichain of the poset. The following lemma is well-known. 

\begin{lemma}~\cite[Section 11.5]{ChanPal2023LinearExtensionsSurvey}\cite[pg. 5]{FelsnerWernisch1997MarkovChains}\label{lem:width_two_as_permutation_poset}
Every naturally labeled poset $P$ of width two is a permutation poset.
\end{lemma}

By combining the previous two lemmata, we have the folowing.

\begin{proposition}\label{prop:Q_W_not_real_rooted}
    Let $Q^{\pi}_{W}(t) = \sum_{\sigma\in [\id, \pi]_W}t^{\des(\pi)}$. 
    There exists $\pi \in \symS_{17}$ such that $Q^{\pi}_{W}(t)$ is not real-rooted.
\end{proposition}

\begin{proof}
    From Lemma~\ref{lem:permutation_poset_linear_extensions}, we have the following equality of polynomials: $Q^{\pi}_{W} = W_{P(\pi)}$. 
    Now, there exists a naturally labeled poset $P$ of width two such that $W_{P}$ is not real-rooted~\cite{Stembridge2007NeggersStanleyCounter}. By Lemma~\ref{lem:width_two_as_permutation_poset}, it follows that there exists $\pi$ such that $P = P_{\pi}$. It follows that $Q^{\pi}_{W}$ is not real-rooted. 
    Concretely, let $\pi \in \symS_{17}$
be defined by
\[
\pi = [2, 4, 6, 8, 10, 1, 12, 3, 15, 5, 17, 7, 9, 11, 13, 14, 16].
\]
Then, using $\mathtt{Mathematica}$, the polynomial $Q^{\pi}_{W}(t) = \sum_{\sigma\in [\id, \pi]_W}t^{\des(\pi)}$ is equal to 
\begin{equation}\label{eq:stembridge}
3 t^8+86 t^7+658 t^6+1946 t^5+2534 t^4+1420 t^3+336 t^2+32 t+1,    
\end{equation}
which has a pair of non-real roots at $-1.85884 \pm 0.14976i$.
\end{proof}

In~\cite{AlexanderssonJal2024RookMatroids}, the so-called poset -- skew shape correspondence was used to prove that $W_{P}$ is ultra-log-concave for naturally labeled width two posets $P$. This together with Lemma~\ref{lem:width_two_as_permutation_poset} and Proposition~\ref{prop:Q_W_not_real_rooted} makes the class of general permutation posets a good candidate for further investigation in this vein. 

\begin{conjecture}\label{conjecture:Q_w_is_ULC}
    Let $\pi \in \symS_{n}$. Then $Q_{W}^{\pi}$ is ultra-log-concave. 
\end{conjecture}

Note that this is a special case of Brenti's conjecture~\cite[Conjecture 1.1]{Brenti1989UnimodalAMS} rephrased in the language of the weak order. 
\bigskip

For width two posets, the $P$-Eulerian polynomial is closely related to the peak-generating polynomial of $\mathcal{L}(P)$ (see \cite[Proposition 1.1]{Stembridge2007NeggersStanleyCounter}). Based on this relation, we can modify $\pi$ in (\ref{eq:stembridge}) and obtain
\[
\widetilde{\pi} = [3, 5, 7, 9, 11, 2, 13, 4, 16, 6, 18, 8, 10, 12, 14, 15, 17, 18],
\]
such that the peak-generating polynomial 
\[
\sum_{\sigma \leq_W \widetilde{\pi}} t^{\peak(\sigma)}
\]
is also equal to \eqref{eq:stembridge}, and hence not real-rooted.
\bigskip

Recall that an \defin{excedance} of a permutation $\sigma$, is an instance $i$ where $\sigma_i > i$. The number of excedances of a permutation is denoted $\exc(\sigma)$.
It is well-known that the classical Eulerian polynomials can also be defined via the excedance statistic:
$A_n(t) = \sum_{\sigma \in \symS_n} t^{\exc(\sigma)}$.
It is therefore natural to ask if excedances also gives rise to real-rooted polynomials when considering intervals in the 
Bruhat order. Computer experiments suggests the following conjecture.
\begin{conjecture}
Let $\pi$ be 312-avoiding. Then
\begin{equation}
\sum_{\sigma \leq_B \pi} t^{\exc(\sigma)}
\end{equation}
is real-rooted.
\end{conjecture}
Again, we cannot extend this conjecture to general $\pi$, 
as the polynomial
\begin{equation}
\sum_{\sigma \leq_B [4, 1, 5, 6, 8, 2, 3, 7]} t^{\exc(\sigma)} = 1 + 21 t + 140 t^2 + 290 t^3 + 127 t^4 + 5 t^5
\end{equation}
has some non-real roots.

\subsection{Multiset rook-Eulerian polynomials}

We now consider the setting of rook placements on boards wherein each row contains exactly one rook, 
but a single column can now contain more than one rook. The context for this set-up is Simion's multiset Eulerian polynomials \cite{Simion1984}, which were also shown to be real-rooted. These polynomials have since been refined and generalized in various directions; see \cite{MaPan2023} for instance. We offer a generalization that considers arbitrary Ferrers~boards. 

Let $\alpha = (\alpha_1,\dotsc,\alpha_k)$ be non-negative integers with total sum $n$,
and let $\lambda$ and $\mu$ be integer partitions such that $\lambda_i > \mu_i$ for all $i$.
We let $\mathcal{W}(\lambda/\mu,\alpha)$ be all words with $\alpha_i$ entries equal to $i$,
such that 
\[
   \mu_i < w_i \leq \lambda_i \quad \text{ for all $i=1,2,\dotsc,n$}.
\]
The vector $\alpha$ is the \defin{content} of each of the words in $\mathcal{W}(\lambda/\mu,\alpha)$. 
The partitions $\lambda/\mu$ are usually called a \defin{skew Ferrers board}.

\begin{example}\label{ex:multiset}
Consider all words with content $\alpha = (2,2,1)$ and the Ferrers board $\lambda=22233$. 
There are in total 30 words with the given content (permutations of $11223$), 
but only $12$ of these fit on the board $\lambda$:
\begin{align*}
\substack{\ytableaushort{{\rook}{\,},{\rook}{\,},{\,}{\rook},{\,}{\rook}{\,},{\,}{\,}{\rook}} \\ \asc(11223)=2 }
&&
\substack{\ytableaushort{{\rook}{\,},{\rook}{\,},{\,}{\rook},{\,}{\,}{\rook},{\,}{\rook}{\,}} \\ \asc(11232)=2 } 
&&
\substack{\ytableaushort{{\rook}{\,},{\,}{\rook},{\rook}{\,},{\,}{\rook}{\,},{\,}{\,}{\rook}} \\ \asc(12123)=3 } 
&&
\substack{\ytableaushort{{\rook}{\,},{\,}{\rook},{\rook}{\,},{\,}{\,}{\rook},{\,}{\rook}{\,}} \\ \asc(12132)=2 } 
&&
\substack{\ytableaushort{{\rook}{\,},{\,}{\rook},{\,}{\rook},{\rook}{\,}{\,},{\,}{\,}{\rook}} \\ \asc(12213)=2 } 
&&
\substack{\ytableaushort{{\rook}{\,},{\,}{\rook},{\,}{\rook},{\,}{\,}{\rook},{\rook}{\,}{\,}} \\ \asc(12231)=2 } 
\\
\substack{\ytableaushort{{\,}{\rook},{\rook}{\,},{\rook}{\,},{\,}{\rook}{\,},{\,}{\,}{\rook}} \\ \asc(21123)=2 } 
&&
\substack{\ytableaushort{{\,}{\rook},{\rook}{\,},{\rook}{\,},{\,}{\,}{\rook},{\,}{\rook}{\,}} \\ \asc(21132)=1 } 
&&
\substack{\ytableaushort{{\,}{\rook},{\rook}{\,},{\,}{\rook},{\rook}{\,}{\,},{\,}{\,}{\rook}} \\ \asc(21213)=2 } 
&&
\substack{\ytableaushort{{\,}{\rook},{\rook}{\,},{\,}{\rook},{\,}{\,}{\rook},{\rook}{\,}{\,}} \\ \asc(21231)=2 } 
&&
\substack{\ytableaushort{{\,}{\rook},{\,}{\rook},{\rook}{\,},{\rook}{\,}{\,},{\,}{\,}{\rook}} \\ \asc(22113)=1 } 
&&
\substack{\ytableaushort{{\,}{\rook},{\,}{\rook},{\rook}{\,},{\,}{\,}{\rook},{\rook}{\,}{\,}} \\ \asc(22131)=1 }
\end{align*}
\end{example}

Let $\alpha$ be a non-negative integer vector with total sum $n$.
The \defin{multiset rook-Eulerian polynomial} $R(\lambda/\mu,\alpha;t)$ is defined as
\begin{equation}
    R(\lambda/\mu,\alpha;t) \coloneqq \sum_{w \in \mathcal{W}(\lambda/\mu,\alpha)} t^{\asc(w)}.
\end{equation}

\begin{example}
The board $\lambda = 22233$ and $\alpha=(2,2,1)$ as in Example~\ref{ex:multiset} gives 
\[
  R(\lambda/\mu,\alpha;t) = t^3+8 t^2+3 t.
\]
\end{example}

As before, we introduce a refined version of the ascent-generating polynomials,
based on the first entry of the words:
\begin{equation}
    R_j(\lambda/\mu,\alpha;t) \coloneqq \sum_{\substack{w \in \mathcal{W}(\lambda/\mu,\alpha) \\ w_1 = j}} t^{\asc(w)}.
\end{equation}
It follows immediately that 
\begin{equation}
   R(\lambda/\mu,\alpha;t) = \sum_j R_j(\lambda/\mu,\alpha;t).
\end{equation}

The refined multiset Eulerian polynomials satisfy the recursion below.
Note that in this case, we do not insert an extra column in the diagram as in Lemma~\ref{lem:Q_lambda_recursion}.
We only prepend an additional first row.
\begin{lemma}
 Let $\lambda = (\lambda_1,\dotsc,\lambda_n)$ and let $\lambda^+ = (m,\lambda_1,\dotsc,\lambda_n)$
 where $m \leq \lambda_1$. Suppose $i \leq m$ and let $\alpha^+  = \alpha + \mathbf{e}_i$. 
 Then we have the recursion
\begin{equation}
 R_i(\lambda^+,\alpha^+;t) = \sum_{j=1}^{i} R_j(\lambda,\alpha;t) + t \sum_{j=i+1}^{\lambda_1} R_j(\lambda,\alpha;t).
\end{equation}
\end{lemma}
\begin{proof}
 This is the same argument as in Lemma~\ref{lem:Q_lambda_recursion}.
\end{proof}

The following conjecture generalizes a result by Simion~\cite[Sect. 2]{Simion1984}.
\begin{conjecture}\label{thm:multisetMain}
For Ferrers boards, the polynomial $R(\lambda,\alpha;t)$ is real-rooted. 
Moreover,
\begin{equation}
  R_{\lambda_1}(\lambda,\alpha;t), R_{\lambda_1-1}(\lambda,\alpha;t), \dotsc, R_2(\lambda,\alpha;t), R_1(\lambda,\alpha;t)
\end{equation}
forms an interlacing sequence.
\end{conjecture}
In the case when $\lambda$ is a rectangular Ferrers board, a result of Simion asserts that $R(\lambda,\alpha;t)$ is real-rooted for any fixed $\alpha$ \cite{Simion1984}.
The interlacing property in the rectangle case is covered by the more recent paper by Ma and Pan~\cite[Theorem 1.11]{MaPan2023}.
See also \cite[Eq. (4.7)]{BrandenHaglundVisontaiWagner2011} for a stable multivariate generalization of Simion's~result.

\begin{remark}
We cannot hope to obtain real-rootedness when extending the multiset setting to skew shapes. For the board $\lambda/\mu = 333321/11$ and content $\alpha = (2,2,2)$, 
we find that $R(\lambda/\mu,\alpha;t) = t^4+4 t^3+6 t^2+t$, which has some non-real roots.
\end{remark}

\section{Acknowledgements}
We thank Petter Bränd\'{e}n and Katharina Jochemko for their feedback and helpful comments. The second author acknowledges support from the Wallenberg AI, Autonomous Systems and Software Program funded by the Knut and Alice Wallenberg Foundation. 

\bibliographystyle{alphaurl}
\bibliography{bibliography}

\begin{thebibliography}{BHVW11}

\bibitem[AJ24]{AlexanderssonJal2024RookMatroids}
Per Alexandersson and {Aryaman} { Jal}.
\newblock Rook matroids and log-concavity of {$P$}-{E}ulerian polynomials.
\newblock {\em arXiv preprint arXiv:2308.14372}, 2024.

\bibitem[AN21]{AlexanderssonNabawanda2021}
Per Alexandersson and Olivia Nabawanda.
\newblock Peaks are preserved under run-sorting.
\newblock {\em Enumerative Combinatorics and Applications}, 2(1), June 2021.
\newblock URL: \url{http://ecajournal.haifa.ac.il/Volume2022/ECA2022_S2A2.pdf},
  \href {https://doi.org/10.54550/ECA2022V2S1R2}
  {\path{doi:10.54550/ECA2022V2S1R2}}.

\bibitem[BHVW11]{BrandenHaglundVisontaiWagner2011}
Petter Br{\"{a}}nd{\'{e}}n, James Haglund, Mirk{\'{o}} Visontai, and David~G.
  Wagner.
\newblock Proof of the monotone column permanent conjecture.
\newblock In {\em Notions of Positivity and the Geometry of Polynomials}, pages
  63--78. Springer Basel, 2011.
\newblock \href {https://doi.org/10.1007/978-3-0348-0142-3_5}
  {\path{doi:10.1007/978-3-0348-0142-3_5}}.

\bibitem[Br{\"{a}}15]{Branden2015}
Petter Br{\"{a}}nd{\'{e}}n.
\newblock Unimodality, log-concavity, real–rootedness and beyond.
\newblock In {\em Handbook of Enumerative Combinatorics}, pages 437--483.
  Chapman and Hall/{CRC}, March 2015.
\newblock \href {https://doi.org/10.1201/b18255-10}
  {\path{doi:10.1201/b18255-10}}.

\bibitem[Bre89]{Brenti1989UnimodalAMS}
Francesco Brenti.
\newblock Unimodal, log-concave and {P}\'{o}lya frequency sequences in
  combinatorics.
\newblock {\em Mem. Amer. Math. Soc.}, 81(413):viii+106, 1989.
\newblock \href {https://doi.org/10.1090/memo/0413}
  {\path{doi:10.1090/memo/0413}}.

\bibitem[CP23]{ChanPal2023LinearExtensionsSurvey}
Swee~Hong Chan and Igor Pak.
\newblock Linear extensions of finite posets.
\newblock {\em arXiv preprint arXiv:2311.02743}, 2023.

\bibitem[Dev06]{Develin2006RookEquiv}
Mike Develin.
\newblock Rook poset equivalence of {F}errers boards.
\newblock {\em Order}, 23(2-3):179--195, 2006.
\newblock \href {https://doi.org/10.1007/s11083-006-9039-8}
  {\path{doi:10.1007/s11083-006-9039-8}}.

\bibitem[EvW04]{EhrenborgVanWilligenburgFerrersGraphs2004}
Richard Ehrenborg and Stephanie van Willigenburg.
\newblock Enumerative properties of {F}errers graphs.
\newblock {\em Discrete Comput. Geom.}, 32(4):481--492, 2004.
\newblock \href {https://doi.org/10.1007/s00454-004-1135-1}
  {\path{doi:10.1007/s00454-004-1135-1}}.

\bibitem[Fro10]{Frobenius1910}
G~Frobenius.
\newblock {\"{U}}ber die {B}ernoullischen und die {E}ulerschen {P}olynome.
\newblock {\em Sitzungsberichte der Preussische Akademie der Wissenschaften},
  pages 809--847, 1910.

\bibitem[FW97]{FelsnerWernisch1997MarkovChains}
Stefan Felsner and Lorenz Wernisch.
\newblock Markov chains for linear extensions, the two-dimensional case.
\newblock In {\em Proceedings of the {E}ighth {A}nnual {ACM}-{SIAM} {S}ymposium
  on {D}iscrete {A}lgorithms {N}ew {O}rleans, {LA}, 1997}, pages 239--247. ACM,
  New York, 1997.

\bibitem[MP23]{MaPan2023}
Jun Ma and Kaiying Pan.
\newblock $(m, i)$-multiset {E}ulerian polynomials.
\newblock {\em Advances in Applied Mathematics}, 149:102547, August 2023.
\newblock URL: \url{http://dx.doi.org/10.1016/j.aam.2023.102547}, \href
  {https://doi.org/10.1016/j.aam.2023.102547}
  {\path{doi:10.1016/j.aam.2023.102547}}.

\bibitem[Sim84]{Simion1984}
Rodica Simion.
\newblock A multiindexed sturm sequence of polynomials and unimodality of
  certain combinatorial sequences.
\newblock {\em Journal of Combinatorial Theory, Series A}, 36(1):15--22,
  January 1984.
\newblock \href {https://doi.org/10.1016/0097-3165(84)90075-x}
  {\path{doi:10.1016/0097-3165(84)90075-x}}.

\bibitem[Sj{\"o}07]{Sjostrand2007}
Jonas Sj{\"o}strand.
\newblock Bruhat intervals as rooks on skew {F}errers boards.
\newblock {\em Journal of Combinatorial Theory, Series A},
  114(7):1182–--1198, October 2007.
\newblock URL: \url{http://dx.doi.org/10.1016/j.jcta.2007.01.001}, \href
  {https://doi.org/10.1016/j.jcta.2007.01.001}
  {\path{doi:10.1016/j.jcta.2007.01.001}}.

\bibitem[Ste07]{Stembridge2007NeggersStanleyCounter}
John~R. Stembridge.
\newblock Counterexamples to the poset conjectures of {N}eggers, {S}tanley, and
  {S}tembridge.
\newblock {\em Trans. Amer. Math. Soc.}, 359(3):1115--1128, 2007.
\newblock \href {https://doi.org/10.1090/S0002-9947-06-04271-1}
  {\path{doi:10.1090/S0002-9947-06-04271-1}}.

\bibitem[SV15]{SavageVisontai2015}
Carla~D. Savage and Mirk{\'{o}} Visontai.
\newblock The $s$-{E}ulerian polynomials have only real roots.
\newblock {\em Transactions of the American Mathematical Society},
  367(2):1441--1466, October 2015.
\newblock \href {https://doi.org/10.1090/s0002-9947-2014-06256-9}
  {\path{doi:10.1090/s0002-9947-2014-06256-9}}.

\bibitem[SY12]{SongYan2012DescentsFerrers}
Chunwei Song and Catherine Yan.
\newblock Descents of permutations in a {F}errers board.
\newblock {\em Electron. J. Combin.}, 19(1):Paper 7, 17, 2012.
\newblock \href {https://doi.org/10.37236/14} {\path{doi:10.37236/14}}.

\end{thebibliography}

\end{document}